\documentclass[a4paper, oneside,11 pt]{article}

\usepackage[a4paper, left=3cm, right=3cm, top=2.5cm, bottom=2.5cm]{geometry}

\usepackage[latin1]{inputenc}
\usepackage[T1]{fontenc}
\usepackage[english]{babel}

\usepackage[all]{xy}
\usepackage{verbatim}
\usepackage{amsmath}
\usepackage{amsthm}
\usepackage{graphicx}
\usepackage{amssymb}
\usepackage{epstopdf}
\usepackage{url}
\usepackage{amsfonts}
\usepackage{todonotes}
\newtheorem{thm}{Theorem}
\newtheorem{lem}[thm]{Lemma}

\newtheorem{prop}[thm]{Proposition}

\def\be{\begin{equation}}
\def\ee{\end{equation}}
\def\bea{\begin{eqnarray}}
\def\eea{\end{eqnarray}}

\usepackage{fancyhdr}

\newtheorem{defn}[thm]{Definition}

\newcommand{\R}{\mathbb{R}}

\numberwithin{thm}{section}
\numberwithin{equation}{section}

\title{Exponential convergence for ultrafast diffusion equations with log-concave weights}

\begin{document}

\author{
Max Fathi\thanks{Universit\'e Paris Cit\'e and Sorbonne Universit\'e, CNRS, Laboratoire Jacques-Louis Lions and Laboratoire de Probabilit\'es, Statistique et Mod\'elisation, F-75013 Paris, France\\
and DMA, \'Ecole normale sup\'erieure, Universit\'e PSL, CNRS, 75005 Paris, France \\
and Institut Universitaire de France\\ Email: \textsf{mfathi@lpsm.paris}},
\,\,Mikaela Iacobelli\thanks{ETH Zurich. Email: \textsf{mikaela.iacobelli@math.ethz.ch}}
%   \and
% Francesco Saverio Patacchini
% \thanks{Department of Mathematics, Imperial College London, South Kensington Campus, London SW7 2AZ, UK. Email: \textsf{f.patacchini13@imperial.ac.uk}}
%    }
}

\maketitle

\begin{abstract}
We study the asymptotic behavior of a weighted ultrafast diffusion PDE on the real line, with a log-concave and log-lipschitz weight, and prove exponential convergence to equilibrium. This result goes beyond the compact setting studied in \cite{Iac1}.
This equation is motivated by the gradient flow approach to the problem of quantization of measures introduced in \cite{CGI1}.
\end{abstract}

\section{Introduction}

In recent years, the asymptotic analysis of solutions to nonlinear parabolic equations has gained significant attention, particularly in connection with gradient flows and entropy methods. See for example \cite{AGS, Jun, Ott} for relevant background. 
This paper examines convergence to equilibrium for the evolution equation  
\begin{equation}
\label{PDE:f}
\partial_t f(t,x) = -r\, \partial_x\Big( f(t,x) \,\partial_x\Big(\frac{\rho(x)}{f(t,x)^{r+1}}\Big) \Big), \quad (t,x) \in (0, \infty) \times \mathbb{R},
\end{equation}
where $r > 1$, $ \rho > 0 $, and $ f(t, \cdot) $ is a probability density on $\mathbb{R}$. When $\rho \equiv 1$, this equation belongs to the class of ultrafast diffusion equations, which have been extensively studied in \cite{Va2}.  

We focus on the range $r > 1$,\footnote{The choice for this range is due to the link with the quantization problem, as explained later, but our results are valid in the full ultrafast diffusion range $r>0.$} a regime that has received less attention due to the fact that, because of the instantaneous mass loss \cite[Theorem 3.1]{Vazquez}, solutions do not exist already in the model case $\rho \equiv 1 $. In contrast, for $r \in (-1,0) $, corresponding to the fast and very fast diffusion cases, solutions have been thoroughly analyzed in various contexts (see, e.g., \cite{Daskalopoulos1997, Carrillo2002, Bonforte2006, Bonforte2007, Blanchet2009, Salvarani2009, Bonforte2010, Bonforte2017a, Bonforte2017b}).  

To ensure well-posedness in our setting, it is crucial that $ \rho $ has finite mass and satisfies appropriate decay conditions. Our work explores these assumptions and their role in the existence and properties of solutions.

As discussed in \cite{Iac1,Iac2,Golse2018}, this equation is motivated by the \emph{quantization problem}. Based on the results in \cite{CGI1, GL, IPS, Iac1}, the long-time behavior of solutions is expected to be
$$
f(t,x) \longrightarrow m(x) := \gamma \rho^{1/(r+1)}(x) \quad \text{as } t \to \infty, \quad \text{where } \gamma := \frac{1}{\int_0^1 \rho(y)^{1/(r+1)} \, dy}.
$$
The convergence for the problem on $[0,1]$ with Neumann boundary conditions was shown in \cite{CGI1} under the assumption that $\rho$ is close to 1 in $C^2$, and extended in \cite{Iac1} under weaker assumptions. Here, we study the same problem on the whole line $\mathbb{R}$, under strong decay assumptions on $\rho$.

As in previous works, our analysis begins with observing that the equation above is the gradient flow of the functional
$$
\mathcal{F}_\rho[f] := \int \frac{\rho(x)}{f(x)^r} \, dx,
$$
with respect to the $W_2$ distance. Setting $$V := -\frac{1}{r+1} \log \rho-\log\gamma$$
we have $\rho = \gamma^{-(r+1)}e^{-(r+1)V}$ and the minimizing measure for $\mathcal{F}_\rho$ is
$m = e^{-V}.$

We consider initial data $f_0$ comparable to $m$, i.e., $f_0 \in \mathcal{P}_{c,C} := \{g \in \mathcal{P}(\mathbb{R}); \, c m \leq g \leq C m \}$ for some constants $0 < c < C$. 
Under this assumption, the existence and uniqueness of a solution $f(t)$ that belongs to $\mathcal{P}_{c,C}$ for all times can be proved thanks to \cite{IPS} and an approximation argument (see Appendix~\ref{app:A}). Thus, our focus is on the asymptotic behavior of solutions.
Our main result is the following:

\begin{thm}\label{thm:1}
Assume that $V$ is convex and $L$-Lipschitz, $V'$ is $L$-Lipschitz, that $m$ has unit mass, and that $f_0 \in \mathcal{P}_{c,C}$. Then there exist constants $A, a > 0$, depending only on $V$,  $c$, and $C$, such that
$$
\int \bigg| \frac{f(t)}{m} - 1 \bigg|^2 m \, dx \leq A e^{-a t}.
$$
\end{thm}

This is a non-compact analogue to \cite[Theorem 1.1]{Iac1}. The values of $A$ and $a$ could be made explicit, and $a$ has a polynomial dependence on $c$ and $C$. Our result disappointingly does not include the case of a Gaussian invariant measure, even though this measure has better concentration properties than those considered here, and should in some sense be closer to the compactly-supported case studied in \cite{Iac1}. It would be interesting to relax our assumptions, and in particular remove the Lipschitz assumption on $V$, so as to cover the Gaussian setting. 

The proof will be based on the gradient flow structure of the PDE, a local convexity estimate for $\mathcal{F}_\rho$ with respect to the Wasserstein distance, and some boundedness properties of Wasserstein geodesics. It is this last ingredient that is specific to dimension one, and our proof shall rely on certain Lipschitz regularity estimates for optimal transport maps, that are expected to fail in higher dimension (with known counterexamples when $V(x) = |x|^2$). It is natural to expect exponential convergence to equilibrium in higher dimension and for small enough perturbations of the invariant measure, but a proof would require some new ingredients.

\section{Proofs}
\subsection{Proof of Theorem~\ref{thm:1}}

Defining $u=\frac{f}{m}$, the dissipation of $\mathcal{F}_\rho$ along the flow is given by 
$$
\frac{d}{dt} \mathcal{F}_\rho[f] = -r^2 \int u |\partial_x u^{-(r+1)}|^2 m =: -I_\rho[f].
$$
The goal is to show the existence of a constant $A(c, C)$ (which shall depend on $V$) such that 
\begin{equation} \label{EPE}
\mathcal{F}_\rho[f] - \mathcal{F}_\rho[m] \leq A(c, C) I_\rho[f], \quad \forall f \in \mathcal{P}_{c,C}.
\end{equation}
Given this inequality, applying Gronwall's lemma yields exponential convergence of $\mathcal{F}_\rho[f] - \mathcal{F}_\rho[m]$ to zero. As discussed in \cite[Section 3]{Iac1}, $\mathcal{F}_\rho[f] - \mathcal{F}_\rho[m]$ is comparable to the $L^2$ distance on $\mathcal{P}_{c,C}$:
$$
\mathcal{F}_\rho[f] - \mathcal{F}_\rho[m] \sim \int \bigg|\frac{f}{m} - 1\bigg|^2 m \, dx, \quad \forall f \in \mathcal{P}_{c,C}.
$$
Thus, proving \eqref{EPE} yields exponential convergence to equilibrium in the $L^2$ sense.

Linearizing \eqref{EPE} for small perturbations of $m$ reveals that a Poincar\'e inequality for $m$ is a necessary condition for \eqref{EPE} to hold.

\begin{defn}
The probability measure $m$ satisfies a Poincar\'e inequality with constant $C_P$:
$$
\operatorname{Var}_m(g) := \int g^2 m \, dx - \left(\int g m \, dx\right)^2 \leq C_P \int |\nabla g|^2 m \, dx, \quad \forall g.
$$
\end{defn}

Since $V$ is assumed to be convex, it follows from \cite{Bobkov1999} that $m=e^{-V}$ satisfies a Poincar\'e inequality for some constant $C_P>0$ (which can be bounded using only the variance of $m$). Of course, many non-convex potentials also satisfy a Poincar\'e inequality, and our convexity assumption shall also be used elsewhere.

We shall combine a (local) semi-convexity property of $\mathcal{F}_\rho$ on $\mathcal{P}_{c,C}$ with respect to the Wasserstein structure and the Poincar\'e inequality to derive \eqref{EPE}.

\begin{prop} \label{prop_convexity}
Let $\rho = e^{-(r+1)V}$ and $0 < c < C < \infty$. Assume that $V$ is convex, $L$-Lipschitz, $V'$ is $L$-Lipschitz, and that $f \in \mathcal{P}_{c,C}$. Then there exists a (computable) constant $\lambda$ (depending on $V$, $c$, and $C$) such that $\mathcal{F}_\rho$ is $(-\lambda)$-convex along the $W_2$-geodesic linking $m$ and $f$. 
\end{prop}
We refer to Appendix \ref{appendix_ot} for a brief introduction to optimal transport geometry, where the notion of Wasserstein geodesic is explained. Note that this statement is not simply a Hessian lower bound within $\mathcal{P}_{c,C}$, since this subspace is not necessarily geodesically convex for the $W_2$ structure. 

The easiest setting for using convexity estimates to prove convergence to equilibrium is when the convexity parameter is non-negative. This is not the case here, so we must also use other ingredients. Assuming the validity of Proposition \ref{prop_convexity} and 
using the convexity property, we obtain the inequality
\begin{equation} \label{hwi}
\mathcal{F}_\rho[f] - \mathcal{F}_\rho[m] \leq W_2(f, m) \sqrt{I_\rho[f]} + \frac{\lambda}{2} W_2(f,m)^2.
\end{equation}
This is analogous to the inequality $f(x) - f(y) \leq \nabla f(x) \cdot (x-y) + \frac{\lambda}{2} |x-y|^2$ for $(-\lambda)$-convex functions in $\mathbb{R}^d$.
See \cite[Corollary 20.13]{Vil}. The proof there is under a global semi-convexity assumptions, and not just along a given geodesic, but the important property for establishing an HWI inequality between two given measures is that the convexity holds along a Wasserstein geodesic linking those two measures. 

We then use the following lemma: 
\begin{lem} \label{lem_local_wi}
Assume $f \in \mathcal{P}_{c,C}$. Then 
$$
W_2^2(f,m) \leq \frac{4 C_P^2 C^{2r+1}}{r^2 (r+1)^2} I_\rho[f].
$$
\end{lem}

\begin{proof}
From \cite[Theorem 1]{Pey}, we know $W_2(f, m) \leq 2 \| u-1 \|_{H^{-1}(m)}$, where $u=f/m$. Moreover, 
\begin{align*}
\| u-1 \|_{H^{-1}(m)} &= \sup \bigg\{\int h(u-1) m \, dx \, \big| \, \int |\nabla h|^2 m \, dx \leq 1 \bigg\} \\
&\leq \sup_{\| \nabla h \|_{L^2(m)} \leq 1} \operatorname{Var}_m(h)^{1/2} \operatorname{Var}_m(u)^{1/2} \\
&\leq C_P \bigg(\int |\nabla u|^2 m \, dx\bigg)^{1/2}.
\end{align*}
Assuming $0 < c \leq u \leq C$, we estimate
$$
\int |\nabla u|^2 m \, dx \leq \frac{C^{2r+3}}{(r+1)^2} \int u |\nabla(u^{-(r+1)})|^2 m \, dx = \frac{C^{2r+3}}{r^2(r+1)^2} I_\rho[f].
$$
Thus, for $f \in \mathcal{P}_{c,C}$, we obtain
$$
W_2^2(f,m) \leq \frac{4 C_P^2 C^{2r+3}}{r^2 (r+1)^2} I_\rho[f].
$$
\end{proof}
Combining Lemma~\ref{lem_local_wi} with \eqref{hwi}, we have
$$
\mathcal{F}_\rho[f] - \mathcal{F}_\rho[m] \leq \bigg(\frac{2C_P C^{r+3/2}}{r(r+1)} + \frac{2\lambda C_P^2 C^{2r+3}}{r^2 (r+1)^2}\bigg) I_\rho[f], \quad \forall f \in \mathcal{P}_{c,C}.
$$
This implies exponential decay of $\mathcal{F}_\rho[f] - \mathcal{F}_\rho[m]$ for $f \in \mathcal{P}_{c,C}$ with a rate depending on $c$, $C$ and $\lambda$, as well as $V$ through $C_P$, concluding the proof of \eqref{EPE}.

To conclude, it remains to prove Proposition~\ref{prop_convexity}, which establishes the semi-convexity property of $\mathcal{F}_\rho$.

\subsection{Proof of Proposition~\ref{prop_convexity}}

\begin{prop}
\label{prop:convex}
Assume that $V$ is convex and $L$-Lipschitz, and that $V'$ is $L$-Lipschitz. If $f \in \mathcal{P}_{c,C}$, then the $W_2$-Hessian of $\mathcal{F}_\rho$ at $f$ is bounded from below by a constant $-\lambda$, which depends only on $c$, $C$, and $V$.
\end{prop}

\begin{proof}
We start from the computations in \cite[Section 4.1]{Iac1} and aim to show that, for any test function $\phi$, we have
\begin{equation} \label{eq_hessian_bound}
2(r+1) \int \frac{\rho'}{f^r} \phi' \phi'' \, dx 
+ r(r+1) \int \frac{\rho}{f^r} (\phi'')^2 \, dx 
+ \int \frac{\rho''}{f^r} (\phi')^2 \, dx 
\geq -\lambda \int (\phi')^2 f \, dx.
\end{equation}
This inequality is shown in \cite[Section 4.1]{Iac1} to express a Hessian lower bound for $\mathcal{F}_\rho$ with respect to the $W_2$-structure. Plugging $\rho = e^{-V}$ into the left-hand side, we have
\begin{align*}
2(r+1) \int \frac{\rho'}{f^r} \phi' \phi'' \, dx 
&= 2(r+1) \int -V' \phi' \phi'' \frac{\rho}{f^r} \, dx, \\
\int \frac{\rho''}{f^r} (\phi')^2 \, dx 
&= \int \big((V')^2 - V''\big) (\phi')^2 \frac{\rho}{f^r} \, dx.
\end{align*}
Combining terms, the left-hand side of \eqref{eq_hessian_bound} becomes
\begin{align*}
2(r+1) \int -V' \phi' \phi'' \frac{\rho}{f^r} \, dx 
+ r(r+1) \int (\phi'')^2 \frac{\rho}{f^r} \, dx 
+ \int \big((V')^2 - V''\big) (\phi')^2 \frac{\rho}{f^r} \, dx.
\end{align*}
Using Young's inequality, we bound the first term:
\begin{align*}
2(r+1) \int -V' \phi' \phi'' \frac{\rho}{f^r} \, dx 
&\geq -r(r+1) \int (\phi'')^2 \frac{\rho}{f^r} \, dx 
- \frac{r+1}{r} \int (V')^2 (\phi')^2 \frac{\rho}{f^r} \, dx.
\end{align*}
Thus, the left-hand side of \eqref{eq_hessian_bound} is bounded from below by
$$
-\int \frac{(V')^2}{r} (\phi')^2 \frac{\rho}{f^r} \, dx 
- \int V'' (\phi')^2 \frac{\rho}{f^r} \, dx.
$$
Using the bound $\rho / f^r \leq c^{-(r+1)} f$, we have
$$
-\int \frac{(V')^2}{r} (\phi')^2 \frac{\rho}{f^r} \, dx 
- \int V'' (\phi')^2 \frac{\rho}{f^r} \, dx 
\geq -c^{-(r+1)} \bigg(\frac{L^2}{r} + L\bigg) \int (\phi')^2 f \, dx.
$$
Setting $\lambda = c^{-(r+1)} \big(\frac{L^2}{r} + L\big)$ completes the proof.
\end{proof}

To apply this result in our situation and derive the inequality \eqref{hwi}, we need to prove that the $W_2$-geodesic linking $m$ and $f$ belongs to $\mathcal P_{c',C'}$ for some universal constants $c',C'>0$. This is the goal of the rest of this section.

\begin{lem} \label{lem_ratio_isop}
Assume $V$ is convex and $L$-Lipschitz. There exists a constant $C_V > 0$, depending only on $V$, such that
$$
C_V^{-1} \leq \frac{e^{-V(x)}}{\min\bigg(\int_{-\infty}^x e^{-V(y)} \, dy, \int_x^\infty e^{-V(y)} \, dy\bigg)} \leq C_V.
$$
\end{lem}

This lemma gives a two-sided bound on the isoperimetric profile of the measure $e^{-V}$. Since $V$ is convex, the lower bound actually only depends on the variance of the measure \cite[Theorem 1]{Bobkov1999}, and does not require the Lipschitz regularity. However, we shall also use the upper bound, which does not hold for every log-concave probability measures (for example, it fails for the standard Gaussian measure). 

\begin{proof}
The function $\frac{e^{-V(x)}}{\min(\cdots)}$ is locally bounded, so it suffices to prove the claim for large $|x|$. Without loss of generality, consider $x \to \infty$ (so that in particular the minimum in the denominator is $\int_x^\infty{e^{-V}}$). 

Let $x_0 > 0$ be such that $\int_{x_0}^\infty e^{-V(y)} \, dy < 1/2$ and $V'(x_0) > 0$. Then, for $x \geq x_0$, since $V'$ is increasing, 
$$
\int_x^\infty e^{-V(y)} \, dy \leq \frac{1}{V'(x_0)} \int_x^\infty V'(y) e^{-V(y)} \, dy = \frac{e^{-V(x)}}{V'(x_0)}.
$$
For the upper bound, since $0 < V'(x) \leq L$ for $x > x_0$, we have
$$
\int_x^\infty e^{-V(y)} \, dy \geq \frac{1}{L} \int_x^\infty V'(y) e^{-V(y)} \, dy = \frac{e^{-V(x)}}{L}.
$$
This proves the desired result.
\end{proof}

\begin{lem} \label{lip_TO_1d}
Assume that $V$ is convex and $L$-Lipschitz, and that $f \in \mathcal{P}_{c,C}$. Then the optimal transport map from $m$ to $f$ satisfies
$$
\frac{c}{CC_V^2} \leq T' \leq \frac{CC_V^2}{c}
$$
where $C_V$ is the constant from Lemma \ref{lem_ratio_isop}. 
\end{lem}

The upper bound is a Lipschitz estimate on the optimal transport map $T$, while the lower bound can also be viewed as a Lipschitz bound on the reverse map $T^{-1}$. Lipschitz estimates on optimal transport maps is a well-studied topic, with nonetheless many open problems in the multi-dimensional setting. An influential first result was Caffarelli's contraction theorem \cite{Caff01}, which states that the Brenier map from a standard Gaussian onto a uniformly convex measure is globally Lipschitz. Various extensions have been considered in \cite{Kol, CFJ, CFS}, yet none apply to a multidimensional version of our setting (bounded perturbations of a nice non-compactly supported measure). Indeed, it is shown in \cite[Appendix]{CFJ} that the Brenier map from a standard multivariate Gaussian onto a bounded perturbation of it may fail to be globally Lipschitz. And uniform Lipschitz regularity of the Brenier map from a standard Gaussian onto a log-lipschitz perturbation of it is currently an open problem \cite[Conjecture 1]{FMS}. However, in dimension one, stronger regularity properties are expected, and can be proved using the explicit representation of transport maps using cumulative distribution functions (a tool that is not available in higher dimension). 
 
\begin{proof}
As per the computations in \cite[Section 4.3]{Mil}, to prove that $T$ is Lipschitz (which is the upper bound), it is enough to compare the (one-sided linear) isoperimetric profiles. That is, given a probability density $g$, we define
$$
Is_g(t) = g \circ G^{-1}(t); \quad G(x) = \int_{-\infty}^x g(s) \, ds,
$$
and to get a $C$-Lipschitz estimate for the optimal transport map from $m$ to $f$, it is enough to show that $Is_f \geq C^{-1} Is_m$. From Lemma~\ref{lem_ratio_isop}, we have that 
$$
C_V \min(t, 1-t) \geq Is_m(t) \geq C_V^{-1} \min(t, 1-t).
$$
Moreover, 
\begin{multline*}
\frac{c}{C} \frac{e^{-V(x)}}{\min\left( \int_{-\infty}^x e^{-V(y)} \, dy, \int_x^{\infty} e^{-V(y)} \, dy \right)} \leq \frac{f(x)}{\min\left( \int_{-\infty}^x f(y) \, dy, \int_x^{\infty} f(y) \, dy \right)} \\
\leq \frac{C}{c} \frac{e^{-V(x)}}{\min\left( \int_{-\infty}^x e^{-V(y)} \, dy, \int_x^{\infty} e^{-V(y)} \, dy \right)}.
\end{multline*}
Thus,
$$
\frac{c}{C_V C} \min(t, 1-t) \leq Is_f(t) \leq \frac{C C_V}{c} \min(t, 1-t).
$$
Hence, 
$$
\frac{c}{CC_V^2} Is_m(t) \leq Is_f(t) \leq \frac{CC_V^2}{c} Is_m(t),
$$
and the map $T$ is therefore $CC_V^2/c$-Lipschitz. From a similar argument, we see that the reverse map from $f$ to $m$ is also $CC_V^2/c$-Lipschitz, and hence $T' \geq c/(CC_V^2)$. 
\end{proof}

\begin{lem} \label{bounded_displacement_TO_1d}
Assume that $V$ is convex and that $f \in \mathcal{P}_{c,C}$. Then the optimal transport map from $m$ to $f$ satisfies
\begin{equation}\label{eq:T(x)}
|T(x) - x| \leq A(c, C, V),
\end{equation}
where $A(c, C, V)$ is a computable constant that only depends on $c$, $C$, and $V$.
\end{lem}

\begin{proof}
We write $f = u e^{-V}$. From the Monge-Amp\`ere equation, we have
$$
e^{-V(x)} = u(T(x)) e^{-V(T(x))} T'(x).
$$
Without loss of generality, assume that $V$ attains its minimum at the origin.

Assume first that $x$ and $y$ have the same sign, say $x, y \geq 0$. We claim that there are constants $\ell_1, \ell_2 > 0$ (depending only on $V$) such that
$$
|V(y) - V(x)| \geq \ell_1 |y - x| - \ell_2 \quad \forall x, y \geq 0.
$$
Indeed, there exists $x_0$ and $\ell_1$ such that $V'(x_0) \geq \ell_1 > 0$. If both $x$ and $y$ are greater than $x_0$, the inequality is trivial by monotonicity of $V'$. If both are smaller than $x_0$, it follows by taking $\ell_2$ large enough. If $x > x_0 > y \geq 0$, then
$$
|V(y) - V(x)| = V(x) - V(y) \geq V(x) - V(x_0) \geq \ell_1 |x - x_0| \geq \ell_1 |x - y| - \ell_1 x_0.
$$

Now, distinguish the cases where $T(x) \geq x$ and $T(x) < x$. First assume $T(x) \geq x$. Since $c \leq u \leq C$ and by Lemma~\ref{lip_TO_1d}, $T'(x) \leq CC_V^2/c$, it follows from the Monge-Amp\`ere PDE that $e^{V(T(x)) - V(x)} \leq C^2C_V^2/c$. Since $V(T(x)) \geq V(x)$, we obtain
$$
|T(x) - x| \leq \frac{\ell_2 + \log(C^2C_V^2/c)}{\ell_1}.
$$

If $T(x) \leq x$, using a similar argument we get $e^{V(x) - V(T(x))} \leq \frac{CC_V^2}{c^2}$, and thus
$$
|x - T(x)| \leq \frac{\ell_2 + \log(CC_V^2/c^2)}{\ell_1}.
$$
This argument shows that when $x$ and $T(x)$ are non-negative, $|T(x) - x| \leq A$. The same reasoning can be adapted to cover the case where both $x$ and $T(x)$ are non-positive. The only remaining case is when $x$ and $T(x)$ have opposite signs.

Let us temporarily assume that $T(0) \geq 0$. By monotonicity of the transport, if $x > 0$ then $T(x) > T(0) \geq 0$, so $x$ and $T(x)$ have the same sign, so we are back to the previous case, and we conclude that
$$
|T(x) - x| \leq A \quad \text{for all} \, x \geq 0.
$$
Consider $x' \leq T^{-1}(0)\leq 0$. Then $T(x') \leq T(T^{-1}(0)) = 0$, meaning that $x'$ and $T(x')$ have the same sign. Once again, we are back to the rpevious case, and we obtain
$$
|T(x') - x'| \leq A \quad \text{for all} \, x' \leq T^{-1}(0).
$$
Applying these estimates at $x = 0$ and $x = T^{-1}(0)$, we conclude
\begin{equation}\label{eq:T(0)}
|T(0)| \leq A, \quad |T^{-1}(0)| \leq A.
\end{equation}
Finally, if $x \in [T^{-1}(0), 0]$, then $T(x) \in [0, T(0)]$, and using \eqref{eq:T(0)} we have
$$
|x - T(x)| \leq 2A.
$$
If $T(0) < 0$, we can adapt the argument by distinguishing the three cases $x \leq 0$, $x \geq T^{-1}(0)$ and $x \in [0, T^{-1}(0)]$. 
\end{proof}

\begin{prop}
Assume that $V$ is convex and $L$-Lipschitz. Let $f \in \mathcal{P}_{c,C}$. There exist computable constants $c', C'$ that depend only on $c$, $C$, and $V$ such that the $W_2$-geodesic between $f$ and $m$ remains in $\mathcal{P}_{c', C'}$.
\end{prop}

\begin{proof}
Let $T$ be the optimal transport map from $m$ to $f$. The $W_2$ geodesic is given by the sequence of densities $f_\alpha$ that are the pushforward of $m$ by $x \mapsto \alpha x + (1-\alpha) T(x)$. Let $u_\alpha = f_\alpha / m$. The goal is to show that the $u_\alpha$ are uniformly bounded from above and below.

We have the Monge-Amp\`ere equation:
$$
e^{-V(x)} = u_{\alpha}((1-\alpha)x + \alpha T(x)) e^{-V((1-\alpha)x + \alpha T(x))} ((1-\alpha) + \alpha T'(x)).
$$
Since $(1-\alpha) + \alpha T'(x)$ is uniformly bounded from above and below (from Lemma~\ref{lip_TO_1d}), it is enough to get uniform bounds on $\exp(V((1-\alpha)x + \alpha T(x)) - V(x))$. Since $V$ is Lipschitz and $T(x) - x$ is uniformly bounded (from Lemma~\ref{bounded_displacement_TO_1d}), this follows immediately. \end{proof}

We now obtain the validity of Proposition~\ref{prop_convexity}. Indeed, by the results above, we know that if $f \in \mathcal P_{c,C}$, then the $W_2$-geodesic linking $m$ and $f$ belongs to $\mathcal P_{c',C'}. $ This allows us to apply Proposition~\ref{prop:convex} with $f_\alpha$ in place of $f$ and $c',C'$ in place of $c,C$ to deduce a uniform lower bound on the $W_2$-Hessian of $\mathcal F_\rho$ at $f_\alpha$, for all $\alpha \in [0,1]$. This proves the desired $(-\lambda)$-convexity, concluding the proof.

\appendix
\section{On the existence and uniqueness of solutions}\label{app:A}
Let $f_0 \in \mathcal{P}_{c,C} := \{g \in \mathcal{P}(\mathbb{R}); \, c m \leq g \leq C m \}$ for some constants $0 < c < C$. 
Our claim is that there exists a unique solution $f(t)$ to \eqref{PDE:f} such that $f(t) \in \mathcal{P}_{c,C}$ for all $t \geq 0.$

This fact has been proved in \cite{IPS} under much weaker assumptions and in arbitrary dimensions on the density $m$, but only in the case of bounded domain. 
While it would be possible to carefully adapt the arguments from \cite{IPS} to our situation, it is easier to argue by approximation.

More precisely, given $k\geq 1$, let consider the potential 
$$
V_k:=\left\{
\begin{array}{ll}
a_k V& \text{in }[-k,k],\\
+\infty& \text{otherwise,}
\end{array}
\right.
$$
where $a_k >0$ is a normalization constant chosen so that $m_k:=e^{-V_k}$ is a probability density. Also, we consider the initial datum
$$
f_{0}^k:=\left\{
\begin{array}{ll}
b_k f_0& \text{in }[-k,k],\\
0& \text{otherwise,}
\end{array}
\right.
$$
where $b_k >0$ is a normalization constant chosen so that $f_{0,k}$ is a probability density. 
Note that $a_k,b_k \to 1$ as $k\to \infty$, hence
$$
f_{0}^k \in \mathcal{P}^k_{c_k,C_k} := \{g \in \mathcal{P}([-k,k]); \, c_k m_k \leq g \leq C_k m_k \},
$$
where $c_k \to c$ and $C_k \to C$ as $k\to \infty$.

Now, since $m_k$ is compactly supported, we can apply \cite[Theorem 1.2]{IPS} to guarantee the existence and uniqueness of a unique solution $f^k(t)\in \mathcal{P}^k_{c_k,C_k}$ starting from $f_0^k$ and solving the PDE corresponding to the equilibrium measure $m_k$, set inside the domain $[-k,k]$, with Neumann boundary conditions.

We now observe that, given any interval $[-R,R]$, for $k>R$ 
we can apply \cite[Lemma 3.3]{IPS} to obtain the compactness of the solutions $f^k(t)$ inside the domain $[-R,R]$.
By a diagonal argument this guarantees that the curve of functions $t\mapsto f^k(t)$ converge (up to a subsequence) to $t\mapsto f(t)\in \mathcal P_{c,C}$, where $f(0)=f_0$ and $f(t)$ solves \eqref{PDE:f}.

Also, the bounds provided by \cite[Theorem 3.5]{IPS} to the solutions $f^k$ imply in the limit that also $f$ satisfy such estimates. Thanks to these bounds, one can apply \cite[Proposition 3.6]{IPS} to prove the $L^1$ contractivity of solutions. In particular, uniqueness holds.

Finally, it is worth noticing that, in our setting, the density $m$ is strictly positive and bounded on compact set, so also $f(t)\in \mathcal P_{c,C}$ is uniformly bounded away from zero and infinity on compact sets. This implies that \eqref{PDE:f} is uniformly parabolic on compact sets, so classical parabolic regularity theory applies to our solution.

\section{Optimal transport} \label{appendix_ot}

We introduce here a few notions and facts from optimal transport theory that are used in this work. 

\begin{defn}[$W_2$ distance]
Let $\mu$ and $\nu$ be two probability measures on $\R^d$ with finite second moment. The Wasserstein distance $W_2$ between $\mu$ and $\nu$ is defined as
$$W_2(\mu, \nu):= \inf_{C(\mu, \nu)} \int{|x-y|^2d\pi}$$
where $C(\mu, \nu)$ is the set of couplings between $\mu$ and $\nu$, that is probability measures on $\R^d \otimes \R^d$ whose first marginal is $\mu$ and second marginal is $\nu$.
\end{defn}

When $\mu$ is absolutely continuous with respect to the Lebesgue measure, Brenier's theorem states that the unique optimal coupling is of the form $(\operatorname{id}, \nabla \varphi) \# \mu$, where $\varphi$ is a convex function. In dimension one, this optimal transport is explicit, and given by $F_{\nu}^{-1} \circ F_\mu$, where $F_\mu$ is the cumulative distribution of $\mu$. 

\begin{defn}[Wasserstein geodesic and convexity]
A Wasserstein geodesic between $\mu_0$ and $\mu_1$ is a curve of probability measures $(\mu_t)_{t \in [0,1]}$ such that $\forall s, t \in [0, 1]$ we have
$$W_2(\mu_s, \mu_t) = |s-t|W_2(\mu_0, \mu_1).$$

A functional $\mathcal{F} : \mathcal{P}_2(\R^d) \longrightarrow \R \cup \{+\infty\}$ is geodesically $\lambda$-convex if it is $\lambda$-convex along any $W_2$ geodesic. 
\end{defn}

If $\pi$ is an optimal coupling between $\mu$ and $\nu$, then the pushforward of $\pi$ by $(x,y) \longrightarrow (1-t)x + ty$ induces a curve of probability measures that is a $W_2$-geodesic. In particular, when $\mu$ is absolutely continuous, the unique geodesic is given by the pushforward of $\mu$ by $(x,y) \longrightarrow (1-t)x + t\nabla \varphi(x)$, where $\nabla \varphi$ is the Brenier map. 

In practice, Otto calculus \cite{Ott} allows to compute (in a formal way) the second derivative of a functional along a $W_2$ geodesic, allowing to check if a (nice enough) functional is $\lambda$-convex. It is this calculus that leads to checking $\lambda$-convexity via \eqref{eq_hessian_bound} in \cite{Iac1}. 

\section*{Acknowledgments}
M. Fathi has received support under
the program "Investissement d'Avenir" launched by the French Government
and implemented by ANR, with the reference ANR-18-IdEx-0001
as part of its program "Emergence".  He was also supported by the Agence Nationale de la Recherche (ANR) Grant ANR-23-CE40-0003 (Project CONVIVIALITY). M. Iacobelli acknowledges the support of the NSF Grant DMS-1926 and the SNSF Starting Grant TMSGI2\_226018.


\begin{thebibliography}{99}

\bibitem{AGS}
L. Ambrosio, N. Gigli, and G. Savar\'e, \emph{Gradient flows in metric spaces and in the space of probability measures}, 2nd ed., Lectures in Mathematics ETH Z\"urich, Birkh\"auser, Basel, 2008.

\bibitem{Blanchet2009}
A. Blanchet, M. Bonforte, J. Dolbeault, G. Grillo, and J. L. V\'azquez, \emph{Asymptotics of the fast diffusion equation via entropy estimates}, Arch. Ration. Mech. Anal. \textbf{191} (2009), no. 2, 347--385.

\bibitem{Bobkov1999}
S.~G. Bobkov, \emph{Isoperimetric and analytic inequalities for log-concave probability measures}, Ann. Probab. \textbf{27} (1999), no.~4, 1903--1921.

\bibitem{Bonforte2006}
M. Bonforte and J. L. V\'azquez, \emph{Global positivity estimates and Harnack inequalities for the fast diffusion equation}, J. Funct. Anal. \textbf{240} (2006), no. 2, 399--428.

\bibitem{Bonforte2007}
M. Bonforte and J. L. V\'azquez, \emph{Reverse smoothing effects, fine asymptotics, and Harnack inequalities for fast diffusion equations}, Bound. Value Probl. \textbf{2007}, Art. ID 21425, 31 pp.

\bibitem{Bonforte2010}
M. Bonforte and J. L. V\'azquez, \emph{Positivity, local smoothing, and Harnack inequalities for very fast diffusion equations}, Adv. Math. \textbf{223} (2010), no. 2, 529--578.

\bibitem{Bonforte2017a}
M. Bonforte, J. Dolbeault, M. Muratori, and B. Nazaret, \emph{Weighted fast diffusion equations (Part I): Sharp asymptotic rates without symmetry and symmetry breaking in Caffarelli-Kohn-Nirenberg inequalities}, Kinet. Relat. Models \textbf{10} (2017), no. 1, 33--59.

\bibitem{Bonforte2017b}
M. Bonforte, J. Dolbeault, M. Muratori, and B. Nazaret, \emph{Weighted fast diffusion equations (Part II): Sharp asymptotic rates of convergence in relative error by entropy methods}, Kinet. Relat. Models \textbf{10} (2017), no. 1, 61--91.

\bibitem{BF} 
M. Bonforte and A. Figalli, \emph{The Cauchy-Dirichlet problem for the fast diffusion equation on bounded domains}, Nonlinear Anal. \textbf{239} (2024), Paper No. 113394, 55 pp.

\bibitem{Caff01} L. A. Caffarelli, 
Monotonicity properties of optimal transportation and the FKG and related inequalities. 
Commun. Math. Phys. 214, No. 3, 547-563 (2000); erratum ibid. 225, No. 2, 449-450 (2002). 

\bibitem{CGI1}
E. Caglioti, F. Golse, and M. Iacobelli, \emph{A gradient flow approach to quantization of measures}, Math. Models Methods Appl. Sci. \textbf{25} (2015), 1845--1885.

\bibitem{CGI2}
E. Caglioti, F. Golse, and M. Iacobelli, \emph{Quantization of measures and gradient flows: a perturbative approach in the 2-dimensional case}, Preprint.

\bibitem{CFS} G. Carlier, A. Figalli and F. Santambrogio, 
On Optimal Transport Maps Between 1 /d-Concave Densities. 
Preprint, arXiv:2404.05456 (2024). 

\bibitem{Carrillo2002}
J. A. Carrillo, C. Lederman, P. A. Markowich, and G. Toscani, \emph{Poincar\'e inequalities for linearizations of very fast diffusion equations}, Nonlinearity \textbf{15} (2002), no. 3, 565--580.

\bibitem{CS}
J. A. Carrillo and D. Slep\c cev, \emph{Example of a displacement convex functional of first order}, Calc. Var. Partial Differential Equations \textbf{36} (2009), 547--564.

\bibitem{CFJ} M. Colombo, A. Figalli and Y. Jhaveri, 
Lipschitz changes of variables between perturbations of log-concave measures. 
Ann. Sc. Norm. Super. Pisa, Cl. Sci. (5) 17, No. 4, 1491-1519 (2017). 

\bibitem{Daskalopoulos1997}
P. Daskalopoulos and M. Del Pino, \emph{On nonlinear parabolic equations of very fast diffusion}, Arch. Ration. Mech. Anal. \textbf{137} (1997), no. 4, 363--380.

\bibitem{E}
J. R. Esteban, A. Rodr\'iguez, and J. L. V\'azquez, \emph{A nonlinear heat equation with singular diffusivity}, Comm. Partial Differential Equations \textbf{13} (1988), 985--1039.

\bibitem{FMS} M. Fathi, D. Mikulincer and Y. Shenfeld, 
Transportation onto log-Lipschitz perturbations. 
Calc. Var. Partial Differ. Equ. 63, No. 3, Paper No. 61, 25 p. (2024). 

\bibitem{GL}
S. Graf and H. Luschgy, \emph{Foundations of Quantization for Probability Distributions}, Lecture Notes in Math., vol. 1730, Springer, Berlin, 2000.

\bibitem{Golse2018}
F. Golse, \emph{Quantization of probability densities: a gradient flow approach}, in \emph{From particle systems to partial differential equations}, Springer Proc. Math. Stat. \textbf{258} (2018), 33--52.

\bibitem{Iac1}
M. Iacobelli, \emph{Asymptotic analysis for a very fast diffusion equation arising from the 1D quantization problem}, Discrete Contin. Dyn. Syst. \textbf{39} (2019), 4929--4943.

\bibitem{Iac2}
M. Iacobelli, \emph{A Gradient Flow Perspective on the Quantization Problem}, PDE Models for Multi-Agent Phenomena, Springer INdAM Ser., \textbf{vol. 28}, Springer, Cham (2018), 145--165.

\bibitem{IPS}
M. Iacobelli, F. S. Patacchini, and F. Santambrogio, \emph{Weighted ultrafast diffusion equations: from well-posedness to long-time behaviour}, Arch. Ration. Mech. Anal. \textbf{232} (2019), 1165--1206.

\bibitem{Jun} 
A. J\"ungel, Entropy methods for diffusive partial differential equations. Cham: Springer; Bilbao: BCAM -- Basque Center for Applied Mathematics. 

\bibitem{Kol} A. V. Kolesnikov. On Sobolev regularity of mass transport and transportation inequalities. Theory
Probab. Appl., 57(2):243--264, 2013.

\bibitem{Mil}
E. Milman, \emph{Spectral estimates, contractions, and hypercontractivity}, J. Funct. Anal. \textbf{243} (2007), 301--340.

\bibitem{Ott} F. Otto, 
\emph{The geometry of dissipative evolution equations: The porous medium equation. }
Commun. Partial Differ. Equations 26, No. 1-2, 101-174 (2001). 

\bibitem{OV} F. Otto and C. Villani,  Generalization of an inequality by Talagrand and links with the logarithmic Sobolev inequality. 
J. Funct. Anal. 173, No. 2, 361-400 (2000). 

\bibitem{Pey}
R. Peyr\'e, \emph{Comparison between $W_2$ distance and $H^{-1}$ norm, and localisation of Wasserstein distance}, ESAIM Control Optim. Calc. Var. \textbf{24} (2018), 1489--1501.

\bibitem{RV}
A. Rodr\'iguez and J. L. V\'azquez, \emph{A well-posed problem in singular Fickian diffusion}, Arch. Ration. Mech. Anal. \textbf{110} (1990), 141--163.

\bibitem{Salvarani2009}
F. Salvarani and G. Toscani, \emph{The diffusive limit of Carleman-type models in the range of very fast diffusion equations}, J. Evol. Equ. \textbf{9} (2009), no. 1, 67--80.

\bibitem{V}
C. Villani, \emph{Topics in Optimal Transportation}, Graduate Studies in Mathematics, vol. 58, American Mathematical Society, Providence, RI, 2003.


\bibitem{Va}
J. L. V\'azquez, \emph{Smoothing and decay estimates for nonlinear diffusion equations}, Oxford Lecture Series in Mathematics and its Applications, vol. 33, Oxford University Press, New York, 2006.

\bibitem{Va2}
J. L. V\'azquez, \emph{The porous medium equation: Mathematical theory}, Oxford Mathematical Monographs, Clarendon Press, Oxford, 2007.

\bibitem{Vazquez}
J. L. V\'azquez, \emph{Failure of the strong maximum principle in nonlinear diffusion. Existence of needles}, Comm. Partial Differential Equations \textbf{30} (2005), 1263--1303.

\bibitem{Vil} C. Villani, 
Optimal transport. Old and new. 
Grundlehren der Mathematischen Wissenschaften 338. Berlin: Springer (ISBN 978-3-540-71049-3). xxii, 973 p.


\end{thebibliography}
\end{document}